\documentclass[12pt, regno]{amsart}

\usepackage{amssymb}
\usepackage{amscd}
\usepackage{amsfonts}
\usepackage{setspace}
\usepackage{version}
\usepackage{endnotes}
\usepackage{float}
\usepackage{color}
\usepackage{mdframed}

\newtheorem{theorem}{Theorem}[section]
\newtheorem{lemma}[theorem]{Lemma}
\newtheorem{proposition}[theorem]{Proposition}

\newtheorem{corollary}[theorem]{Corollary}

\theoremstyle{remark}
\newtheorem*{remark}{Remark}

\theoremstyle{definition}

\theoremstyle{remark}

\numberwithin{equation}{section}

\newcommand{\ee}{\varepsilon}
\newcommand{\1}{\mathbf{1}}

\newcommand{\dd}{{\rm d}}

\newcommand{\CA}{\mathcal{A}}

\newcommand{\cond}{\operatorname{cond}}

\begin{document}

\title[Bombieri-Vinogradov for multiplicative functions, II]{When does the Bombieri-Vinogradov Theorem hold for a given multiplicative function?}

\author[A. Granville]{Andrew Granville}
\address{AG: D\'epartement de math\'ematiques et de statistique\\
Universit\'e de Montr\'eal\\
CP 6128 succ. Centre-Ville\\
Montr\'eal, QC H3C 3J7\\
Canada; 
and Department of Mathematics \\
University College London \\
Gower Street \\
London WC1E 6BT \\
England.
}
\thanks{A.G.~has received funding from the
European Research Council  grant agreement n$^{\text{o}}$ 670239, and from NSERC Canada under the CRC program.}
\email{{\tt andrew@dms.umontreal.ca}}
\author[X. Shao]{Xuancheng Shao }
\address{Mathematical Institute\\ Radcliffe Observatory Quarter\\ Woodstock Road\\ Oxford OX2 6GG \\ United Kingdom}
\email{Xuancheng.Shao@maths.ox.ac.uk}
\thanks{X.S.~was supported by a Glasstone Research Fellowship.}
 
 \begin{abstract} Let $f$ and $g$ be   $1$-bounded multiplicative functions for which $f*g=1_{.=1}$. The Bombieri-Vinogradov Theorem holds for both $f$ and $g$  
 if and only if  
  the Siegel-Walfisz criterion holds for both $f$ and $g$, and 
  the Bombieri-Vinogradov Theorem  holds for $f$ restricted to the primes.
 \end{abstract}


\date{\today}

\maketitle

\section{Introduction}

\subsection{Background and the main result}
Given an arithmetic function $f$, we define, whenever $(a,q)=1$,
 $$
\Delta(f,x;q,a):=\sum_{\substack{n\leq x \\ n\equiv a \pmod q}} f(n)  -  \frac 1{\varphi(q)} \sum_{\substack{n\leq x \\ (n,q)=1}} f(n),
$$
which, as $a$ varies, indicates how well $f(.)$ is distributed in the arithmetic progressions mod $q$. In many examples it is difficult to obtain a strong bound on $\Delta(f,x;q,a)$ for arithmetic progressions modulo a particular  $q$ but one can perhaps do better on ``average''. We therefore define the following:

\medskip

\noindent \textbf{The Bombieri-Vinogradov Hypothesis for $f(n)$ with $n$ up to $x$.} \textit{For any given $A>0$ there exists a constant $B=B(A)$ such that }
\begin{equation}\label{eq:BVI}
\sum_{\substack{ q\leq \sqrt{x}/(\log x)^B }}  \max_{a:\ (a,q)=1}  | \Delta(f,x;q,a) | \ll_A \frac x {(\log x)^A}.
\end{equation}
\medskip

The Bombieri-Vinogradov Hypothesis for $f$   formulates the idea that $f$ is well-distributed, on ``average'',  in arithmetic progressions with moduli $q$ almost as large as $\sqrt{x}$. It also directly implies that $f$ is well-distributed in arithmetic progressions with small moduli. In particular the following is an immediate consequence:

\medskip

\noindent \textbf{The Siegel-Walfisz criterion for $f(n)$ with $n$ up to $x$.} \textit{For any given $A>0$ and any $(a,q)=1$ we have  }
\begin{equation}\label{eq:SW}
|\Delta(f,x;q,a)| \ll_A  \frac x {(\log x)^A} .
\end{equation}
\medskip

In this article we focus on $1$-bounded multiplicative functions $f$;  that is, those $f$ for which $|f(n)|\leq 1$ for all $n\geq 1$. We define
\[
F(s) = \sum_{n=1}^{\infty} \frac{f(n)}{n^{s}}\ \ \text{and} \ \
- \frac{F'(s)}{F(s)} = \sum_{n=2}^{\infty}   \frac{\Lambda_f(n)}{n^{s}},  
\]
for Re$(s)>1$.  The function $\Lambda_f(.)$ is supported only on prime powers; 
we restrict  attention to the class ${\mathcal C} $ of multiplicative functions $f$ for which
\[
 |\Lambda_f(n)|\leq   \Lambda(n) \ \ \textrm{for all} \ \ n\geq 1.
 \]
This includes  most $1$-bounded multiplicative functions of interest, including all $1$-bounded completely multiplicative functions. Two key observations are that if $f\in \mathcal C$ then each $|f(n)|\leq 1$, and if $f\in\mathcal C$ and $F(s)G(s)=1$ then $g\in \mathcal C$. Here $g$ is the convolution inverse of $f$; that is, $(f*g)(n)=1$ if $n=1$, and $0$ otherwise.

Define ${\mathcal P}$ to be the set of primes, so that the arithmetic function $f\cdot 1_{\mathcal P}$ is the function $f$ but supported only on the primes.  The classical Bombieri-Vinogradov Theorem is, in our language, the Bombieri-Vinogradov Hypothesis for $1_{\mathcal P}$. The Bombieri-Vinogradov Hypothesis holds  trivially for the corresponding multiplicative function $1(.)$. Many of the proofs of the Bombieri Vinogradov theorem (for example, those going through Vaughan's identity)  relate the distribution of $1_{\mathcal P}$ in arithmetic progressions to the distribution of $\mu(.)$ in arithmetic progressions; note that $\mu$ is the convolution inverse of the multiplicative function $1$. This is the prototypical example of the phenomenon we discuss in this article.

Our main question here is  to address for what $f$ does the Bombieri-Vinogradov Hypothesis hold? Evidently 
the  Siegel-Walfisz criterion must hold for $f$, but what else is necessary? In \cite[Proposition 1.4]{GSh}, we exhibited an $f\in \mathcal C$ for which the  Siegel-Walfisz criterion holds, and yet \eqref{eq:BVI} fails for any $A>1$ and any $B$. The key feature in our   construction of $f$ was that  the Bombieri-Vinogradov Hypothesis did not hold for $f\cdot 1_{\mathcal P}$. As we now see in our main result, this is also   a necessary condition:

\begin{theorem}\label{thm: MathTheorem}  Suppose that $f, g\in \mathcal C$ with $F(s)G(s)=1$. \hfill  \break
\indent \emph{(a)}\ If the Bombieri-Vinogradov Hypothesis holds for both $f$ and $g$ then the Bombieri-Vinogradov Hypothesis holds for $f\cdot 1_{\mathcal P}$; and \hfill  \break
\indent \emph{(b)}\ If the Bombieri-Vinogradov Hypothesis holds for $f\cdot 1_{\mathcal P}$ and  the Siegel-Walfisz criterion holds for $f$ then the Bombieri-Vinogradov Hypothesis holds for   $f$.
 \end{theorem}

Since $g\cdot 1_{\mathcal P} = - f\cdot 1_{\mathcal P}$, this can all be expressed more succinctly as follows: 
\medskip

 \begin{mdframed} 
Suppose that $f, g\in \mathcal C$ with $f*g=1_{.=1}$. Then
\begin{center} The Bombieri-Vinogradov Hypothesis holds for both $f$ and $g$ \\ 
 if and only if  \\
The Bombieri-Vinogradov Hypothesis holds for $f\cdot 1_{\mathcal P}$, and  \\ 
  the Siegel-Walfisz criterion holds for both $f$ and $g$.
\end{center} 
 \end{mdframed}
\medskip

This kind of ``if and only if'' result in the theory of multiplicative functions bears some similarity to (and inspiration from) 
(1.4) and Theorem 1.2 of \cite{Kou}, and much of the discussion there.
 
\section{More explicit results}

Theorem \ref{thm: MathTheorem} is not  as powerful as it looks at first sight since  it is  of little use if one wishes to prove  \eqref{eq:BVI} for a function $f$ whose definition depends on a particular $x$ (as the hypothesis of Theorem \ref{thm: MathTheorem} makes assumptions for all $x$).  In this section we will give uniform versions of both parts of Theorem \ref{thm: MathTheorem}.

The   Bombieri-Vinogradov Hypothesis for $f\in \mathcal C$   fails if  $f$ is a character of small conductor (for example  $f(n) = (n/3)$), or  ``correlates'' with such a character; that is, the sum 
\[
 S_f(x,\chi):= \sum_{n\leq x} f(n) \overline{\chi}(n) 
\]
 is ``large''. We can take such characters  into account as follows:\  Given any finite set of primitive characters, $\Xi$,   let $\Xi_q$ be the set of characters mod $q$ that are induced by the characters in $\Xi$, and then define  
 \[
\Delta_\Xi(f,x;q,a):=  \sum_{\substack{n\leq x \\ n\equiv a \pmod q}} f(n) -   \frac 1{\varphi(q)} \sum_{ \substack{  \chi  \in   \Xi_q }} \chi(a) S_f(x,\chi) .
\]
Note that $\Delta(f,x;q,a)=\Delta_{ \{ 1\} }(f,x;q,a)$.

\subsection{Precise statement of Theorem~\ref{thm: MathTheorem}(b) and B-V for smooth-supported $f$}  Our uniform version of Theorem~\ref{thm: MathTheorem}(b) is  the following result, from which Theorem~\ref{thm: MathTheorem}(b) immediately follows.

\begin{theorem}\label{thm: mr2}
Fix $A\geq 0,\ B>A+5$ and $\gamma >2A+6$.
Given $x \geq 2$, let   $Q=x^{1/2}/(\log x)^{B}$ and $y = x/(\log x)^{\gamma}$. 
 Suppose that  $f\in \mathcal C$, and assume that 
 \begin{equation}\label{eq:BVI3}
\sum_{\substack{ q\leq Q }}  \max_{a:\ (a,q)=1}  | \Delta(f\cdot 1_{\mathcal P},X;q,a) | \ll \frac X {(\log x)^A \log (x/y)}
\end{equation}
for all $X$ in the range $y\leq X\leq x$; and that
\[
|\Delta(f,X;q,a)| \ll  \frac X {(\log x)^{A+2B}} ,
\]
whenever $(a,q)=1$, for all $X$ in the range $x^{1/2}\leq X\leq x$. 
Then
\[ 
\sum_{q \leq Q} \max_{(a,q)=1} |\Delta(f,x; q,a)|
\ll   \frac x{(\log x)^{A-1}}.
\]
\end{theorem}
 
\begin{remark}
 Theorem \ref{thm: mr2} is stronger the larger we can take $y$ (and thus the smaller we can take $\gamma$), since that reduces the assumptions made of the form 
 \eqref{eq:BVI3}. We have been able to take any $\gamma >2A+6$ in Theorem \ref{thm: mr2}.
 In section \ref{sec: SWnotBV}, we will show that we must have $\gamma \geq {A-3}$.  It would be interesting to know the optimal power, $\gamma$, of $\log x$ that one can take in the definition of $y$.
\end{remark}

An integer $n$ is \emph{$y$-smooth} if all of its prime factors are $\leq y$.
In \cite{GSh} we proved the Bombieri-Vinogradov Hypothesis for $y$-smooth supported 
$f\in \mathcal C$ satisfying the Siegel-Walfisz criterion, provided $y\leq x^{1/2-o(1)}$.  
For arbitrary $f\in \mathcal C$, we may therefore use this result to obtain the Bombieri-Vinogradov Hypothesis for the $f$-values restricted to $y$-smooth $n$, and need  a different approach for those $n$ that have a large prime factor (that is, a prime factor $>y$).

First though, we have been able to develop a rather different method based on ideas of Harper \cite{Har} to significantly extend our range for $y$.

\begin{theorem} \label{Easy Cor 2*} Fix $A\geq 0,\ B>A+5$ and $\gamma >2A+6$.
Given $x \geq 2$, let $Q=x^{1/2}/(\log x)^{B}$ and $y = x/(\log x)^{\gamma}$. Let  $\CA$ be the set of all primitive characters of conductor at most $ (\log x)^{B}$.
  If  $f\in \mathcal C$ is supported only on the $y$-smooth integers then
\[ 
\sum_{q \leq Q} \max_{(a,q)=1} |\Delta_{\CA}(f,x; q,a)|
\ll   \frac x{(\log x)^A}.
\]
\end{theorem}

Proving this result takes up the bulk of this paper, and occupies Sections~\ref{sec:thmb-part1} and~\ref{sec:thmb-part2}. In Section~\ref{sec:thmb-part3} we deduce Theorem~\ref{thm: mr2} from Theorem~\ref{Easy Cor 2*}.
 

 \subsection{Precise statement of Theorem~\ref{thm: MathTheorem}(a)} It is considerably easier to prove the converse result, since one can easily express
 $f\cdot 1_{\mathcal P}$ in terms of a suitably weighted convolution of $f$ and $g$.
 Theorem \ref{thm: MathTheorem}(a)  follows from the next result, taking $\Xi=\{ 1\}$.  

\begin{theorem}\label{thm:fng}  Given $f\in \mathcal C$, define $g\in \mathcal C$ to be that multiplicative function for which $F(s)G(s)=1$. 
Fix $A,C \geq 0$ and $\ee > 0$.  Let $x$ be large, let $2 \leq Q \leq x^{1/2}/(\log x)^{5A/2+C/4+7/2}$ and let $\Xi$ be a set of at most $(\log x)^C$ primitive characters. Suppose that the following properties hold for both $h=f$ and $h=g$: (i) If $(a,q)=1$ then 
\[
\Delta_\Xi(h,X;q,a) \ll \frac{ X} {(\log x)^{14A+28}}
\]
for all $X$ in the range $x^{0.4}\leq X\leq x$; (ii) The B-V type result
\[
 \sum_{\substack{ q\leq Q }}  \max_{a:\ (a,q)=1}  |\Delta_\Xi (h ,X;q,a)|  \ll \frac X{(\log x)^{A+\ee}} 
 \]
holds for all $X$ in the range $x/(\log x)^{6A+C+10}\leq X\leq x$. Then
\begin{equation}\label{eq:BVI2}
\sum_{\substack{ q\leq Q }}  \max_{a:\ (a,q)=1}  | \Delta_\Xi(f\cdot 1_{\mathcal P},x;q,a) | \ll \frac x {(\log x)^{A}} ;
\end{equation}
and the same holds with $f\cdot 1_{\mathcal P}$ replaced by $g\cdot 1_{\mathcal P}$. 
 \end{theorem}
 
The range $x^{0.4} \leq X \leq x$ appearing in property (i) above can be reduced to $x^{1/2}/(\log x)^{15A+C/2+29} \leq X \leq x$, but we are not concerned about this detail. We will prove Theorem~\ref{thm:fng} at the end of Section~\ref{sec:thma}, after establishing a result about convolutions in Section~\ref{sec:convolution}.

\section{The algebra of the Bombieri-Vinogradov Hypothesis}\label{sec:convolution}

We will establish the next main result using ideas from section 9.8 of \cite{FI} (which has its roots in Theorem 0 of \cite{BFI}). We will sketch a proof of  a modification of Theorem 9.17 of \cite{FI}, being a little more precise and correcting a  couple of minor errors. We assume that $f$ and $g$ are arithmetic functions, with $|f(n)|, |g(n)|\leq 1$ for all $n$, but not necessarily multiplicative.  Let $\Xi$ be a set of primitive characters. 
For $h \in \{f,g\}$ we assume that if $(a,q)=1$ then 
\begin{equation} \label{FISW}
\Delta_\Xi(h,N;q,a) \ll \frac{ H(N)   N^{1/2}} {(\log N)^A}
\end{equation}
for some $A \geq 0$, where $H(N):= (\sum_{n\leq N} |h(n)|^2)^{1/2}$. In \cite{FI} they have $\Xi=\{ 1\}$ throughout but the modifications for arbitrary $\Xi$ are straightforward. The first key step is (9.73) in \cite {FI}. One can be a bit more precise (e.g. by choosing $C$ there more precisely) and show that if $\psi$ is a character mod $q$ with $q\leq (\log N)^A$, but $\psi\not\in \Xi_q$, and if \eqref{FISW} holds, then for any positive integer $m$ we have
\[
\sum_{\substack{n\leq N \\ (n,m)=1}} h(n)\psi(n) \ll q^{1/3} \tau(m) \frac{ H(N)   N^{1/2}} {(\log N)^{A/3}}.
\]

Now  for $M,N \geq 2$ let $f_M(m)=f(m)$ if $m\leq M$ and $f_M(m)=0$ if $m>M$, and define $g_N$ similarly.
Assume that $M\leq N$ and that \eqref{FISW} holds for $h=g_N$.
Theorem 9.16 of~\cite{FI} then becomes,\footnote{Note that in the statement of Theorem 9.16 of~\cite{FI}, the authors claim 
to have obtained a power of $ \log MN$ in the denominator of the third term of the upper bound whereas we only claim a power of $\log N$, as in their proof. This makes no difference here since we added the hypothesis that $N\geq M$, but they did not have this hypothesis in \cite{FI}.}
\begin{equation} \label{FISW2}
\begin{split}
\sum_{\substack{ q\leq Q }}  \max_{a:\ (a,q)=1}  | \Delta_\Xi (f_M*g_N ,MN;q,a) | & \ll  \big(   Q+ \sqrt{N}(\log Q)^2\\
&+\frac{ \sqrt{MN} \log Q}{(\log N)^{A/7}}\big) F(M)G(N).
\end{split}
\end{equation}

We now assume that  \eqref{FISW} holds\footnote{In Theorem 9.17 of \cite{FI}, the authors only assume \eqref{FISW} for $h=g$. This is because of their over-optimistic error term in Theorem 9.16. The correction  seems to force one to   assume that \eqref{FISW} holds for $h=f$ as well.}  for both $h=f$ and $h=g$ for any $N$ in the range $\sqrt{x}\leq N\leq x$ and some $A > 7$. For $C \geq 0$ define
\[
(f*g)_C(r) := \sum_{\substack{mn=r \\ m,n\leq x/(\log x)^C}} f(m)g(n) .
 \]
Our version of   Theorem 9.17 of \cite {FI} states the following:

\begin{lemma}\label{lem:FISW3}
Let the notations and assumptions be as above. Let $B = A/7-1$ and $C = 2B+2$. Then
\[
\sum_{\substack{ q\leq Q }}  \max_{a:\ (a,q)=1}  | \Delta_\Xi ((f*g)_C ,x;q,a) |   \ll    \frac x{(\log x)^D} 
\]
for $Q=\sqrt{x}/(\log x)^B$ with $D = (B-3)/2$.
\end{lemma}

\begin{proof}
The proof involves how to partition the values of $m,n$ in the sum defining $(f*g)_C$, so as to apply \eqref{FISW2}. Using the trivial bounds $F(M)\leq \sqrt{M}$ and $G(N)\leq \sqrt{N}$,  \eqref{FISW2} now reads
\begin{equation} \label{FISW4}
\sum_{\substack{ q\leq Q }}  \max_{a:\ (a,q)=1}  | \Delta_\Xi (f_M*g_N ,MN;q,a) |  \ll     \sqrt{Nx}(\log x)^2 
 +\frac{ x  }{(\log x)^{B}}
  \end{equation}
if $MN \ll x$. We apply  \eqref{FISW4} as follows for those $m\leq \sqrt{x}$ (an analogous construction works for those $n\leq \sqrt{x}$, as  well as for any overlap):\ 

First take $M_0=(\log x)^C$ and $N=x/M_0$ in \eqref{FISW4}.
Now let $\Delta=1/(\log x)^{D+2}$. With two applications of \eqref{FISW4}, we obtain \eqref{FISW4} with
$f_M$ replaced by $f_M-f_{(1-\Delta)M}$, and $g_N$ replaced by $g_N$.

We apply this with $M_j=(\log x)^C(1-\Delta)^{-j}$ and $N_j=x/M_j$, for $0\leq j \leq J$, where $J$ is the minimal integer for which $M_J\geq \sqrt{x}$ so that $J \asymp (\log x)/\Delta$. The total contribution here is
\[
\ll \sum_{j=0}^J \left(   (1-\Delta)^{j/2} \frac x{(\log x)^{C/2-2}}   +\frac{ x  }{(\log x)^{B}}\right) 
\ll  \frac{  x  }{\Delta(\log x)^{B-1}}  \ll  \frac{  x  }{ (\log x)^{D}},
\]
where the last inequality follows from our choice of $D$ and $\Delta$.

For each integer $m\in ((1-\Delta)M_j,M_j]$ we have missed out the values of $n$ in the range $(N_j,x/m]$. There are 
$\leq N_j \Delta/(1-\Delta)$ such values of $n$, for each of the $\leq \Delta M_j$ values of $m$ in the interval, a total of
$\ll \Delta^2 x$ pairs for each $j$. Therefore the total number of mixed pairs $m,n$ is 
$\ll J \Delta^2 x \ll \Delta x \log x\ll x/(\log x)^{D+1}$. This completes the proof of Lemma~\ref{lem:FISW3}.
\end{proof}


Of course we are really interested in $f*g$ not $(f*g)_C$, so now we study the difference:
For $y=x/(\log x)^C$ we have 
\begin{align*}
\Delta_\Xi (f*g ,x;q,a) &- \Delta_\Xi ((f*g)_C ,x;q,a)   \\
& = \sum_{\substack{m\leq (\log x)^C \\ (m,q)=1}}  f(m) (\Delta_\Xi (g ,x/m;q,a/m) -  \Delta_\Xi (g ,y;q,a/m) ) \\
& + \sum_{\substack{n\leq (\log x)^C \\ (n,q)=1}} g(n) (\Delta_\Xi (f ,x/n;q,a/n) -  \Delta_\Xi (f ,y;q,a/n) ),
\end{align*}
and so
\begin{align*}
&\sum_{\substack{ q\leq Q }}  \max_{a:\ (a,q)=1}  |\Delta_\Xi (f*g ,x;q,a)| 
 \leq  \sum_{\substack{ q\leq Q }}  \max_{b:\ (b,q)=1}  |\Delta_\Xi ((f*g)_C ,x;q,b) |  \\
&+ \sum_{\substack{m\leq (\log x)^C  }}  \sum_{\substack{ q\leq Q }}  \max_{c:\ (c,q)=1}  |\Delta_\Xi (g ,x/m;q,c)| 
+ \sum_{\substack{m\leq (\log x)^C  }}  \sum_{\substack{ q\leq Q }}  \max_{d:\ (d,q)=1}  |  \Delta_\Xi (g ,y;q,d) | \\
& + \sum_{\substack{n\leq (\log x)^C  }}  \sum_{\substack{ q\leq Q }}  \max_{e:\ (e,q)=1}  |\Delta_\Xi (f ,x/n;q,e)|
+ \sum_{\substack{n\leq (\log x)^C }}  \sum_{\substack{ q\leq Q }}  \max_{k:\ (k,q)=1}  | \Delta_\Xi (f ,y;q,k)| .
\end{align*}
The first term above is handled by Lemma~\ref{lem:FISW3}.
We assume that for $h=f$ and $h=g$ we have 
\[
 \sum_{\substack{ q\leq Q }}  \max_{a:\ (a,q)=1}  |\Delta_\Xi (h ,X;q,a)|  \ll \frac X{(\log x)^D \log\log x} 
 \]
for all $X$ in the range $y\leq X\leq x$, so the above becomes
\[
 \ll    \frac x{(\log x)^D} + \sum_{\substack{m\leq (\log x)^C  }} \frac {x/m}{(\log x)^D \log \log x}  \ll    \frac x{(\log x)^{D}}.
\]
We now summarize what we have proved.

\begin{proposition} \label{prop: Convol2} Fix $D \geq 0$ and let $A=14D+28,\ B=2D+3$ and $C=4D+8$. Let $x$ be large and let $\Xi$ be a set of primitive characters. Let $f$ and $g$ be given arithmetic functions with the following properties: Taking $h=f$ or $h=g$ we have that (i) Each $|h(n)|\leq 1$; (ii) If $(a,q)=1$ then 
\[
\Delta_\Xi(h,X;q,a) \ll \frac{ X} {(\log x)^A}
\]
for all $X$ in the range $x^{1/2}\leq X\leq x$; (iii) The B-V type result
\[
 \sum_{\substack{ q\leq Q }}  \max_{a:\ (a,q)=1}  |\Delta_\Xi (h ,X;q,a)|  \ll \frac X{(\log x)^D \log\log x} 
 \]
holds for all $X$ in the range $x/(\log x)^C\leq X\leq x$, where $2 \leq Q \leq \sqrt{x}/(\log x)^B$. Then
 \begin{equation*}
\sum_{\substack{ q\leq Q }}  \max_{a:\ (a,q)=1}  | \Delta_\Xi(f*g,x;q,a) | \ll \frac x {(\log x)^{D}} .
\end{equation*}
 \end{proposition}

\section{Proof of the uniform version of Theorem~\ref{thm: MathTheorem}(a)}\label{sec:thma}

In this section we deduce Theorem~\ref{thm:fng} from Proposition~\ref{prop: Convol2}, using the identity $\Lambda_f = g*f\log$. We start with two lemmas passing between $f$ and $f\log$.

\begin{lemma}\label{lem: Log-SW} Fix $A \geq 0$. Let $f$ be a given $1$-bounded arithmetic function, let $(a,q)=1$, and suppose we are given a set of primitive characters $\Xi$. If
\begin{equation}\label{eq: SWf}
\Delta_{\Xi}(f, X; q,a) \ll \frac{X}{(\log X)^A}
\end{equation}
for all $X$ in the range $x/(\log x)^A \leq X \leq x$, then
\begin{equation}\label{eq: SWfL}
\Delta_{\Xi}(f\log, X; q,a) \ll \frac{X}{(\log X)^{A-1}}
\end{equation}
for $X = x$. Conversely if~\eqref{eq: SWfL} holds for all $X$ in the range $x/(\log x)^A \leq X \leq x$, then~\eqref{eq: SWf} holds for $X=x$.
\end{lemma}

\begin{proof}
Let $F(n;q,a)=1_{n\equiv a \pmod q} -   \frac 1{\varphi(q)} \sum_{\substack{  \chi  \in   \Xi_q }} \chi(a\overline{n}) $, so that we have
$\Delta_\Xi(f,x;q,a) =  \sum_{n\leq x} f(n)F(n;q,a)$. Moreover
\[
\Delta_\Xi(f\log ,x;q,a)=  \sum_{n\leq x} f(n)F(n;q,a)\log n = \int_{1}^x \log t\, d\Delta_\Xi(f,t;q,a) ,
\]
by the usual technique of partial summation,
so that 
 \begin{equation}\label{eq:f-fL}
 \begin{split}
\Delta_\Xi(f\log ,x;q,a)& - \Delta_\Xi(f\log ,X;q,a)  = \Delta_\Xi(f,x;q,a) \log x 
\\ &- \Delta_\Xi(f,X;q,a) \log X -  \int_{X}^x  \Delta_\Xi(f,t;q,a) \frac{dt}t, 
 \end{split}
 \end{equation}
 where $X=x/(\log x)^A$. By~\eqref{eq: SWf} the three terms on the right hand side above are all $\ll x/(\log x)^{A-1}$. By trivially bounding each $|f(n)|$ by $1$, we obtain
 \[ |\Delta_{\Xi}(f\log, X; q,a)| \leq X \log x \leq \frac{x}{(\log x)^{A-1}}. \]
This yields the first part of the lemma. For the second part we begin with the analogous identity
\[
\Delta_\Xi(f ,x;q,a)=    \int_{2}^x \frac{1}{\log t} \, d\Delta_\Xi(f\log ,t;q,a) ,
\]
and the proof proceeds entirely analogously.
\end{proof}

 \begin{lemma}\label{lem: Log} Fix $A, C \geq 0$. Let $f$  be a given $1$-bounded arithmetic function, $2\leq Q\leq x/(\log x)^{A+C/2+1}$, and suppose we are given a set of primitive characters, $\Xi$, containing $\leq (\log x)^C$ elements. If  
\begin{equation} \label{eq: BVf}
\sum_{\substack{ q\leq Q }}  \max_{a:\ (a,q)=1}  | \Delta_\Xi(f,X;q,a) |  \ll \frac X {(\log X)^A}
\end{equation}
for all $X$ in the range $x/(\log x)^{A+C/2+1}\leq X\leq x$ then
\begin{equation} \label{eq: BVfL}
\sum_{\substack{ q\leq Q }}  \max_{a:\ (a,q)=1}  | \Delta_\Xi(f\log, X;q,a) |  \ll \frac X {(\log X)^{A-1}}
\end{equation}
for $X=x$. Conversely if \eqref{eq: BVfL} holds for all $X$ in the range $x/(\log x)^{A+C/2+1}\leq X\leq x$ then
\eqref{eq: BVf} holds for $X=x$.
\end{lemma}

\begin{proof} 
As before we use~\eqref{eq:f-fL} but now with $X=x/(\log x)^{A+C/2+1}$. This implies that 
 \begin{align*}  
 \sum_{\substack{ q\leq Q }}  & \max_{a:\ (a,q)=1}  | \Delta_\Xi(f\log, x;q,a) | 
  \leq  \sum_{\substack{ q\leq Q }}  \max_{b:\ (b,q)=1}  | \Delta_\Xi(f\log, X;q,b) |  \\
 & +  \sum_{\substack{ q\leq Q }}  \max_{c:\ (c,q)=1} |\Delta_\Xi(f,x;q,c)| \log x
 +  \sum_{\substack{ q\leq Q }}  \max_{d:\ (d,q)=1} |\Delta_\Xi(f,X;q,d)| \log X\\
 &\hskip 1.5in +\int_{X}^x  \sum_{\substack{ q\leq Q }}  \max_{e:\ (e,q)=1} |\Delta_\Xi(f,t;q,e) | \frac{dt}t.
  \end{align*}
 By \eqref{eq: BVf} the last three terms are
 \begin{align*}  
 \ll_A  \frac x {(\log x)^{A-1}} +\int_{X}^x  \frac {dt} {(\log t)^{A}}  \ll_A  \frac x {(\log x)^{A-1}}.
  \end{align*}
Now, by trivially bounding each $|f(n)|$ by 1, we obtain
\[ | \Delta_\Xi(f\log, X;q,b) | \leq (1+|\Xi_q|) \frac{X\log X}q , \]
and so 
 \begin{align*}
 \sum_{\substack{ q\leq Q }}  \max_{b:\ (b,q)=1} &  | \Delta_\Xi(f\log, X;q,b) |  \leq 
  \sum_{\substack{ q\leq Q }} (1+|\Xi_q|) \frac{X\log X}q \\
  & \ll X(\log x)^2 + \sum_{\chi \pmod r \in \Xi}  \  \sum_{\substack{ q\leq Q \\ r|q }}   \frac{X\log X}q\\
  & \ll \left( 1 + \sum_{\chi \pmod r \in \Xi} \frac 1r \right) X(\log x)^2  \ll X(\log x)^{C/2+2}  \ll_A  \frac x {(\log x)^{A-1}}.
\end{align*}
which yields the first part of the lemma. The proof of the second part is again analogous.
\end{proof}

 \begin{proof} [Proof of Theorem \ref{thm:fng}] 
 Let $(\log x)^{-1}f\log$ denote the function $n \rightarrow (\log x)^{-1}f(n)\log n$. By the B-V type assumption on $f$, the first part of Lemma~\ref{lem: Log} implies that 
 \begin{equation*} 
\sum_{\substack{ q\leq Q }}  \max_{a:\ (a,q)=1}  | \Delta_\Xi((\log x)^{-1} f \log,X;q,a) |  \ll \frac X {(\log X)^{A+\ee}},
\end{equation*}
for all $X$ in the range $x/(\log x)^{5A+C/2+9}\leq X\leq x$. The same holds with $(\log x)^{-1} f\log$ above replaced by $g$.

By the S-W type assumption on $f$, the first part of Lemma~\ref{lem: Log-SW} implies that if $(a,q)=1$ then
  \begin{equation*} 
  \Delta_\Xi((\log x)^{-1} f \log,X;q,a)   \ll \frac X {(\log X)^{14A+28}},
\end{equation*}
for all $X$ in the range $x^{0.45}\leq X\leq x$. The same holds with $(\log x)^{-1} f\log$ above replaced by $g$.
 
 The coefficients of  $-F'(s)/F(s)= G(s)\cdot (-F'(s))$ yield the identity 
 $\Lambda_f  =   g* f \log$. Therefore by Proposition \ref{prop: Convol2} applied to the $1$-bounded functions $g$ and $(\log x)^{-1}f\log$, we obtain
\begin{equation*}
\sum_{\substack{ q\leq Q }}  \max_{a:\ (a,q)=1}  | \Delta_\Xi(\Lambda_f,X;q,a) | \ll \frac X {(\log X)^{A-1}} ,
\end{equation*}
for all $X$ in the range $x/(\log x)^{A+C/2+1}\leq X\leq x$.  

The contribution of the prime powers $p^k$ with $k\geq 2$ does not come close to the upper bound, and so 
\begin{equation*}
\sum_{\substack{ q\leq Q }}  \max_{a:\ (a,q)=1}  | \Delta_\Xi((f\cdot 1_{\mathcal P})\log ,X;q,a) | \ll_A \frac X {(\log X)^{A-1}} ,
\end{equation*}
for all $X$ in the range $x/(\log x)^{A+C/2+1}\leq X\leq x$.  Finally the second part of Lemma~\ref{lem: Log} implies 
\eqref{eq:BVI2}. 
  \end{proof}

\section{Factorizing smooth numbers and using the large sieve}\label{sec:thmb-part1}

We develop an idea of Harper \cite{Har} to prove the following result, which will yield rich consequences in the next section.

\begin{proposition}\label{prop:factorize2}
Let $ 2 \leq y \leq x$ be large. Let completely multiplicative $f\in \mathcal C$ be supported on $y$-smooth integers. Let $2 \leq D,Q \leq x$. Then
\[ \sum_{q \leq Q} \frac{1}{\varphi(q)} \sum_{\substack{\chi\pmod{q} \\ \cond(\chi) > D}} |S_f(x,\chi)| \ll 
\left( Q x^{1/2}   + x^{7/8} + \frac{x}{D}\right) (\log x)^{4} + (xy)^{1/2}    (\log x)^{3}. \]
\end{proposition}

To prove this we begin by proving a marginally weaker result (weaker in the sense that one only saves $(x/y)^{1/4}$ from the trivial bound instead of potentially saving $(x/y)^{1/2}$ in Proposition~\ref{prop:factorize2}).

\begin{proposition}\label{prop:factorize}
Let $2 \leq y \leq x$ be large. Let completely multiplicative $f\in \mathcal C$ be supported on $y$-smooth integers. Let $2 \leq D,Q \leq x$. Then
\[ \sum_{q \leq Q} \frac{1}{\varphi(q)} \sum_{\substack{\chi\pmod{q} \\ \cond(\chi) > D}} |S_f(x,\chi)| \ll x^{1/2}\left(Q + (xy)^{1/4} + \frac{x^{1/2}}{D}\right) (\log x)^{5}. \]
\end{proposition}

\begin{proof}
Suppose that $\chi\pmod{q}$ is induced by $\psi\pmod{r}$. Let $h(.)$ be the multiplicative function which is supported only on powers of the primes $p$   which   divide  $q$ but not $r$, and then let $h(p^k)=(g\overline{\psi})(p^k)$ where $g$ is the convolution inverse of $f$. Then 
$f\overline{\chi} = h* f\overline{\psi}$ and so
\[
 S_f(x,\chi)= \sum_{m\geq 1}  h(m) S_f(x/m,\psi). 
 \]
As each $|h(m)|\leq 1$ we deduce that
\[ 
|S_f(x,\chi)| \leq \sum_{\substack{m\geq 1 \\ M| q \\ (M,r)=1}} |S_f(x/m,\psi)|, 
\]
where $M=\prod_{p|m} p$.
Now, we wish to bound
\[ 
\sum_{q \leq Q} \frac{1}{\varphi(q)}\sum_{\substack{r|q \\ r> D}}     \sum_{\substack{\psi\pmod{r}  \text{ primitive} \\ \chi \text{ induced by } \psi}} |S_f(x,\chi)| 
\]
which, writing  $q=rMn$,  is 
\[
\leq 
\sum_{D < r \leq Q}     \sum_{\substack{\psi\pmod r \\ \psi\text{ primitive}}}  \sum_{\substack{m\geq 1 \\  (M,r)=1}} |S_f(x/m,\psi)| \sum_{rMn \leq Q} \frac{1}{\varphi(rMn)};
\]
and this is 
\begin{equation}\label{eq:factorize0} 
\ll (\log Q) \sum_{m \leq x} \frac{1}{\varphi(M)} \sum_{D < r \leq Q} \frac{1} {\varphi(r)}
\sum_{\substack{\psi\pmod r \\ \psi\text{ primitive}}} |S_f(x/m,\psi)|  . 
\end{equation}
Let us first dispose of large values of $m$. For fixed $m$, the sum over $r$ and $\psi$ can be bounded using Cauchy-Schwarz by
\[  \left(\sum_{r \leq Q} \sum_{\substack{\psi\pmod r \\ \psi\text{ primitive}}} \frac{1}{r\varphi(r)}\right)^{1/2} \left(\sum_{r \leq Q} \frac{r}{\varphi(r)} \sum_{\substack{\psi\pmod r \\ \psi\text{ primitive}}} |S_f(x/m,\psi)|^2\right)^{1/2}. \]
The   sum in the first bracket is $\ll \log Q$, and the sum in the second  bracket is $\ll (\frac xm + Q^2)\frac xm$ by the large sieve. Thus the contributions to~\eqref{eq:factorize0} from those $m \geq M_0$, where $M_0 = x/y$, is 
\begin{align*}
& \ll (\log Q)^{3/2} \sum_{M_0 \leq m \leq x} \frac{1}{\varphi(M)} \left(\frac{x}{m} + \frac{Qx^{1/2}}{m^{1/2}}\right) \\
& \ll x^{1/2}\left(\frac{x^{1/2}}{M_0} + \frac{Q}{M_0^{1/2}}\right) (\log Q)^{3/2} \ll (y+Qy^{1/2}) (\log Q)^{3/2},
\end{align*}
which is acceptable. 

Now fix   $m \leq M_0$, and write $X = x/m$ so that $X \geq x/M_0 = y$. 
Set $V_0 = (X/y)^{1/2}$ so that $V_0 \geq 1$. For every $y$-smooth integer $V_0 < n \leq X$, we have a unique factorization $n = uv$ with the properties that
\[ P_+(u)\leq P_-(v), \ \ v > V_0, \ \ v/P_-(v) \leq V_0. \]
This can be achieved by putting prime factors of $n$ into $v$ in descending order, until the size of $v$ exceeds $V_0$ for the first time. Thus
\begin{equation} \label{eq: doublesum}
S_f(X, \psi) = S_f(V_0,\psi) + \sum_{\substack{V_0 < v \leq  yV_0 \\ v/P_-(v) \leq V_0}} \sum_{\substack{u \leq X/v \\ P_+(u) \leq P_-(v)}} f(uv) \overline{\psi}(uv). 
\end{equation}
The $u$-summation has length at least $X/yV_0 = V_0$, which explains the choice of $V_0$. 

Each $|S_f(V_0,\psi)|\leq V_0$, and so the 
  contribution of these terms to \eqref{eq:factorize0} is $\ll   QV_0$ for fixed $m$, giving in total
\[
\ll Q \log Q \sum_{m\leq M_0} \frac{1}{\varphi(M)} \left( \frac{x/m} y\right)^{1/2} \ll \frac{Qx^{1/2} \log Q}{y^{1/2}},
\]
which is again acceptable.

 To analyze the double sum over $u,v$, we first dyadically divide the ranges of $u,v,P_+(u),P_-(v)$. For parameters $U,V,P_+,P_-$ (which can all be taken to be powers of $2$) satisfying
\begin{equation}\label{eq:factorize-ranges} 
U,V \leq X/V_0, \ \ V > V_0, \ \ UV \leq X, \ \ 2 \leq P_+,P_- \leq y,  \ \ P_+ < 2P_-,
\end{equation}
consider the double sum
\begin{equation}\label{eq:factorize2} 
\sum_{\substack{V_0<v \leq yV_0 \\ V \leq v < 2V \\ v/P_-(v) \leq V_0 \\ P_- \leq P_-(v) < 2 P_-}} \sum_{\substack{u \leq X/v \\ U \leq u <  2U \\ P_+(u) \leq P_-(v) \\ P_+ \leq P_+(u) < 2P_+}} f(uv) \overline{\psi}(uv). 
\end{equation}

For the moment, let us pretend that the ``cross conditions" $uv \leq X$ and $P_+(u) \leq P_-(v)$ are not there (for example when $4UV \leq X$ and $2P_+ \leq P_-$), so that the variables $u,v$ are completely separated and~\eqref{eq:factorize2} takes the form
\begin{equation}\label{eq:double-sum} 
\left(\sum_{U \leq u < 2U} a(u) \overline{\psi}(u)\right) \left(\sum_{V \leq v < 2V} b(v) \overline{\psi}(v) \right), \end{equation}
for some $|a(u)| \leq 1$ and $|b(v)| \leq 1$ (which depend on $f$ but not on $\psi$). By Cauchy-Schwarz and then the large sieve inequality, we obtain
\begin{align} \sum_{R < r \leq 2R} & \frac{1}{\varphi(r)} \sum_{\substack{\psi\pmod r\\ \psi\text{ primitive}}} \left|\sum_{U \leq u < 2U} a(u) \overline{\psi}(u) \right| \left|\sum_{V \leq v < 2V} b(v) \overline{\psi}(v) \right|  \nonumber \\ 
& \ll \frac{1}{R} (U^{1/2}+R) (V^{1/2}+R) (UV)^{1/2} .\label{R-Bound}
\end{align}
Summing this up over $R$'s in the interval $[D,Q]$ and the various possibilities given by~\eqref{eq:factorize-ranges}, we get
\begin{align*} 
& \ll \left( \frac{X\log X}D + \frac X {V_0^{1/2}} + QX^{1/2} \log X  \right) (\log y)^2\\
& \ll \left( \frac{X\log X}D + X^{3/4}y^{1/4} + QX^{1/2} \log X  \right) (\log y)^2.
\end{align*}
Summing this over $m$, taking $X=x/m$ we have a contribution
\[ \ll \left( \frac{x}D + x^{3/4}y^{1/4} + Qx^{1/2}  \right) (\log x)^4, \]
to \eqref{eq:factorize0}, as $Q,y\leq x$, which is again acceptable.

To deal with the restrictions $uv \leq X$ and $P_+(u) \leq P_-(v)$ (when necessary), we use Perron's formula in the form
\[ \1_{uv \leq X} = \frac{1}{2\pi i} \int_{1/2-iT}^{1/2+iT} \frac{\widetilde{X}^{s_1}}{u^{s_1}v^{s_1}} \cdot \frac{\dd s_1}{s_1}  + O(x^{-4}) \]
with $\widetilde{X} = \lfloor X \rfloor + 1/2$ and $T=x^5$; and 
\[ \1_{P_+(u) \leq P_-(v)} = \frac{1}{2\pi i} \int_{1/2-iT}^{1/2+iT} \frac{(P_-(v)+1/2)^{s_2}}{P_+(u)^{s_2}} \cdot \frac{\dd s_2}{s_2} + O(x^{-4}). \]
For example, when $UV \asymp X$ and $P_+\asymp P_-$, we can write~\eqref{eq:factorize2} using the above applications of Perron's formula as
\begin{equation}\label{eq:factorize-perron} 
\int_{1/2-iT}^{1/2+iT} \int_{1/2-iT}^{1/2+iT} \left(\sum_{U \leq u < 2U} a(s_1,s_2;u) \overline{\psi}(u)\right) \left(\sum_{V \leq v < 2V} b(s_1,s_2;v) \overline{\psi}(v)\right)  \frac{\dd s_1 \dd s_2}{s_1s_2} + O(x^{-3}), 
\end{equation}
where $a(s_1,s_2;u)$ is supported on those $u$ with $P_+ \leq P_+(u) < 2P_+$ and takes the form
\[ a(s_1,s_2;u) = f(u)\cdot \frac{U^{s_1}}{u^{s_1}} \cdot \frac{P_+^{s_2}}{P_+(u)^{s_2}}, \]
and $b(s_1,s_2;v)$ is supported on those $v$ with $V_0 < v \leq yV_0, v/P_-(v) \leq V_0, P_- \leq P_-(v) < 2P_-$ and takes the form
\[ b(s_1,s_2;v) = f(v) \cdot \frac{\widetilde{X}^{s_1}}{U^{s_1}v^{s_1}} \cdot \frac{(P_-(v)+1/2)^{s_2}}{P_+^{s_2}}.  \]
Note that $|a(s_1,s_2;u)| \ll 1$ and $|b(s_1,s_2;v)| \ll 1$. Thus we can treat the integrand of~\eqref{eq:factorize-perron} just as we did~\eqref{eq:double-sum}. We have two extra powers of $\log T$ which come from integrating $\dd s_1 \dd s_2/|s_1 s_2|$, and we can absorb the errors coming from the $O(x^{-3})$ in~\eqref{eq:factorize-perron} since they are negligible.  We have only one power of $\log y$ arising from the dyadic dissection of $P_-$ and $P_+$. We therefore obtain the upper bound
\[ \ll \left( \frac{X}D + X^{3/4}y^{1/4} + QX^{1/2}  \right) (\log X) (\log T)^2 \log y \]
in total for a given $m$. Summing over $m$ gives a similar contribution to last time (but now with an extra factor of $\log x$), which is equally acceptable.

Finally we have to account for the cases where $4UV <X$ and $P_+\asymp P_-$, and where
$UV \asymp X$ and $2P_+< P_-$. Following the same methods precisely we obtain the same bounds.
 This completes the proof.
\end{proof}

\begin{proof} [Proof of Proposition \ref{prop:factorize2}] Let $g$ be the completely multiplicative function for which $g(p)=f(p)$ when $p\leq 2\sqrt{x}$, and $g(p)=0$ for $p> 2\sqrt{x}$. We apply Proposition \ref{prop:factorize} to obtain
\[ \sum_{q \leq Q} \frac{1}{\varphi(q)} \sum_{\substack{\chi\pmod{q} \\ \cond(\chi) > D}} |S_g(x,\chi)| \ll \left( Q x^{1/2} + x^{7/8} + \frac{x}{D}\right) (\log x)^{5}, \]
an acceptable bound.  We may clearly assume that $y > 2\sqrt{x}$ since otherwise $f=g$.

Now $S_f(x,\chi)=S_g(x,\chi)+S_h(x,\chi)$ where
 $h(n):= f(n)- g(n)$. If $h(n)\ne 0$ then $n=up$ for some prime $p$ in the range $2\sqrt{x}<p\leq y$, and integer $u<\frac 12 \sqrt{x}$, so that $h(n)=f(u)f(p)$. We proceed \emph{analogously} to the proof of Proposition \ref{prop:factorize}, though now $U$ and $V$ are restricted to the range $U\leq  \frac 14 \sqrt{x}$ and $2\sqrt{x}<V\leq y/2$, while $P_+$ and $P_-$ are irrelevant and removed from the argument, so things are significantly simpler. We therefore obtain, for an element of our dyadic partition, the 
same upper bound \eqref{R-Bound}. Summing now over our range for $U$ and $V$ with $UV<X/4$ we obtain the upper bound
\[
\ll \frac{X\log X}D + X^{1/2}y^{1/2} + X x^{-1/4} +  QX^{1/2} \log X .
\]
Finally summing up over $m$ with $X=x/m$ gives an upper bound of
\[
\ll \frac{x(\log x)^2}D + x^{1/2}y^{1/2}\log x + Qx^{1/2} (\log x)^2 .
\]
For the cases in which $UV \asymp X$ we obtain the same upper bound times $(\log T)^2 \ll (\log x)^2$.
Our claimed result follows.
\end{proof}

\section{Consequences of Proposition \ref{prop:factorize2}}\label{sec:thmb-part2}

In this section we deduce Theorem~\ref{Easy Cor 2*}. First we establish a version of Proposition~\ref{prop:factorize2} for $f \in \mathcal C$ that may not be completely multiplicative.

\begin{corollary} \label{Easy Cor 1}
Let  $2 \leq y \leq x$ be large. Let  $f\in \mathcal C$ be supported on the $y$-smooth integers. Let $D\leq  x^{1/3}$ and $Q\leq x/D^2$.  Then
\[ \sum_{q \leq Q} \frac{1}{\varphi(q)} \sum_{\substack{\chi\pmod{q} \\ \cond(\chi) > D}} |S_f(x,\chi)| \ll \left(
\left( Q x^{1/2}   + x^{7/8} + \frac{x}{D}\right) (\log x)^{5} + (xy)^{1/2}    (\log x)^{3} \right)(\log D)^2. \]
\end{corollary}

\begin{proof}  Let $f^*$ be the completely multiplicative function obtained by taking $f^*(p)=f(p)$.  Let $g(p^k)=f(p^k)-f(p)f(p^{k-1})$, so that $g$ is supported only on powerful integers, and $f=g*f^*$. We deduce that 
\[
S_f(x,\chi) = \sum_{n\leq x} g(n) \overline{\chi}(n) S_{f^*}(x/n,\chi),
\]
and so 
\[
|S_f(x,\chi)| \leq  \sum_{n\leq x} |g(n)| \ |S_{f^*}(x/n,\chi)|.
\]
Summing over all $q\leq Q$ we obtain
\[ 
\sum_{q \leq Q} \frac{1}{\varphi(q)} \sum_{\substack{\chi\pmod{q} \\ \cond(\chi) > D}} |S_f(x,\chi)| \leq
\sum_{n\leq x} |g(n)| \ \sum_{q \leq Q} \frac{1}{\varphi(q)} \sum_{\substack{\chi\pmod{q} \\ \cond(\chi) > D}}  |S_{f^*}(x/n,\chi)|.
\]
Set $N = D^2$. For the sum over $n\leq N$ we use Proposition \ref{prop:factorize2}, as  $D,Q\leq x/N$, to obtain the upper bound
\[
\ll \sum_{n\leq N} \frac{|g(n)|}{\sqrt{n}} \ \left( \left( Q x^{1/2}   + x^{7/8} + \frac{x}{D}\right) (\log x)^{5} + (xy)^{1/2}    (\log x)^{3} \right).
\]
Since each $|g(p^k)|\leq 2$ and $g$ is only supported on the powerful, we deduce that 
\[
\sum_{n\leq N} \frac{|g(n)|}{\sqrt{n}} \leq \prod_{p\leq \sqrt{N}} \left( 1 + \frac{2}{p} +   \frac{2}{p^{3/2}} +\ldots \right) \ll (\log N)^2.
\]

It remains to deal with $n>N$: \ The argument near the beginning of the proof of Proposition \ref{prop:factorize} which gave a bound for $m\geq M_0$ can be adjusted here to give a bound when $m\geq 1$, so that
\[
\sum_{q \leq Q} \frac{1}{\varphi(q)} \sum_{\substack{\chi\pmod{q} \\ \cond(\chi) > D}} |S_{f^*}(x,\chi)| 
 \ll \left(x +  Qx^{1/2} \right) (\log Q)^{3/2} .
\]
Therefore the sum over $n>N$ is 
\[
 \ll 
\sum_{N<n\leq x} |g(n)| \  \left( \frac xn +  Q\left( \frac xn \right)^{1/2} \right) (\log Q)^{3/2} \ll
 \frac{x}{\sqrt{N}}  (\log N) (\log x)^{3/2}  + Q x^{1/2} (\log x)^{7/2}.
\]
Taking $N=D^2$ we obtain the claimed result.
\end{proof}

\begin{proof}  [Proof of Theorem \ref{Easy Cor 2*} ] We take $D= (\log x)^{B}$ in Corollary \ref{Easy Cor 1}, noting that 
\[ |\Delta_{\CA}(f,x; q,a)| \leq \frac{1}{\varphi(q)} \sum_{\substack{\chi\pmod q \\ \cond(\chi) > D}} |S_f(x,\chi)| \]
by the definition of $\CA$. Thus Corollary~\ref{Easy Cor 1} implies that the quantity above is
\[ \ll \left(\frac{x}{(\log x)^{B-5}} + \frac{x}{(\log x)^{\gamma/2-3}}\right) (\log D)^2.   \]
This is $\ll x/(\log x)^A$ since $B > A+5$ and $\gamma > 2A+6$.
 \end{proof}

\section{Proof of the uniform version of Theorem~\ref{thm: MathTheorem}(b)}\label{sec:thmb-part3}

 In the section we deduce
Theorem \ref{thm: mr2} from Theorem~\ref{Easy Cor 2*}. We begin by extending Theorem \ref{Easy Cor 2*}  to all $f\in \mathcal C$.

\begin{corollary}\label{thm: conv2}
Fix $A\geq 0,\ B>A+5$ and $\gamma >2A+6$.
Given $x$, let $Q=x^{1/2}/(\log x)^{B}$ and $y = x/(\log x)^{\gamma}$. Let  $\CA$ be the set of all primitive characters of conductor at most $ (\log x)^{B}$.
 Suppose that  $f\in \mathcal C$, and assume that 
\[
\sum_{\substack{ q\leq Q }}  \max_{a:\ (a,q)=1}  | \Delta_{\CA}(f\cdot 1_{\mathcal P},X;q,a) | \ll \frac X {(\log x)^{A}\log(x/y)}
\]
for all $X$ in the range $y\leq X\leq x$.
Then
\[ 
\sum_{q \leq Q} \max_{(a,q)=1} |\Delta_{\CA}(f,x; q,a)|
\ll   \frac x{(\log x)^A}.
\]
\end{corollary}

 \begin{proof}  Let $f_y(p^k)=f(p^k)$ if $p\leq y$, and $f_y(p^k)=0$ otherwise.  If
$f(n)\ne 0$ but $f_y(n)=0$ where $n\leq x$, then $n$ has a prime factor $p>y$, which can only appear in $n$ to the power one, and so 
we can write $n=mp$ with $f(n)=f(m)f(p)$. Moreover $m=n/p< x/y$. Therefore
 \begin{align*}
\Delta_{\CA}(f-f_y,x;q,a)  &= \sum_{\substack{ m \leq x/y \\ (m, q)=1}} f(m)  
\left(  \sum_{\substack{y<p\leq x/m \\ p \equiv a\overline{m}  \pmod q}} f(p)  -  \frac 1{\varphi(q)} \sum_{ \substack{  \chi  \in   \Xi_q }} \chi(a)   \sum_{\substack{y<p\leq x/m  \\ (p,q)=1}} (f\overline{\chi})(p) \right) \\
& = \sum_{\substack{ m \leq x/y \\ (m, q)=1}} f(m)  (  \Delta_{\CA}(f\cdot 1_{\mathcal P},x/m;q,a\overline{m}) -\Delta_{\CA}(f\cdot 1_{\mathcal P},y;q,a\overline{m}) ) .
 \end{align*}
 Summing this up over $q\leq Q$, and as each $|f(m)|\leq 1$, we deduce that 
 \begin{align*}
\sum_{q \leq Q} \max_{(a,q)=1} |\Delta_{\CA}(f,x;q,a)|& \leq \sum_{q \leq Q} \max_{(b,q)=1} |\Delta_{\CA}(f_y,x;q,b)| \\
+\sum_{\substack{ m \leq x/y}}  
\sum_{\substack{ q \leq Q \\ (q,m)=1 }} \max_{(c,q)=1} &  | \Delta_{\CA}(f\cdot 1_{\mathcal P},x/m;q,c) | 
+\sum_{\substack{ m \leq x/y}} 
\sum_{\substack{ q \leq Q \\ (q,m)=1 }} \max_{(d,q)=1}  | \Delta_{\CA}(f\cdot 1_{\mathcal P},y;q,d) |.
\\
 \end{align*}
 We bound the first  term by using Theorem \ref{Easy Cor 2*}, and the other terms using the hypothesis to get an upper bound
 \[
  \ll \frac x{(\log x)^{A}}+\sum_{\substack{ m \leq x/y}} \frac {x/m} {(\log x)^{A}\log(x/y)}\ll \frac x{(\log x)^{A}},
 \]
as claimed. 
 \end{proof}

To use Corollary~\ref{thm: conv2} we need the following result, which follows immediately from the proof of Proposition 3.4 of \cite{GSh}.

\begin{lemma} \label{Using SW} Fix $A,C\geq 0$.   Let $f\in \mathcal C$ be such that 
\[
|\Delta(f,X;q,a)| \ll  \frac X {(\log x)^{A+C}} ,
\]
whenever $(a,q)=1$ for all $X$ in the range $x^{1/2}<X\leq x$. Suppose that $\Xi$ is a set of primitive characters, containing $\ll   (\log x)^{C}$ elements. Let $Q \leq x$, and
for each $q \sim Q$ let $a_q\pmod q$ be a residue class with $(a_q,q) = 1$.  Then
 \[
 \sum_{q \sim Q}   \left|  \Delta_{\Xi}(f,x;q,a_q)  \right|  \ll \frac x{(\log x)^A} 
 \]
if and only if
\[ 
\sum_{q \sim Q} \left|  \Delta(f,x;q,a_q)  \right|    \ll \frac x{(\log x)^A}. 
\]
\end{lemma}

\begin{proof} [Proof of Theorem \ref{thm: mr2}] The hypothesis of Corollary \ref{thm: conv2} holds, and so
\[ 
\sum_{q \leq Q} \max_{(a,q)=1} |\Delta_{\CA}(f,x; q,a)|
\ll   \frac x{(\log x)^A}.
\]
Note that 
$|\CA|\asymp (\log x)^{2B}$ so we may take $C=2B$ in Lemma~\ref{Using SW} for each dyadic range of $q$ to deduce our result.
\end{proof}

 \section{$f$ which satisfy the Siegel-Walfisz criterion but not the Bombieri-Vinogradov Hypothesis} \label{sec: SWnotBV}
 
In this section we justify the remark following the statement of Theorem~\ref{thm: mr2}. The proof of Theorem \ref{thm: mr2} allows us to replace  \eqref{eq:BVI3} in the hypothesis by 
  \begin{equation}\label{eq:BVI4}
\sum_{\substack{ q\leq Q }}  \max_{a:\ (a,q)=1}  | \Delta(f\cdot 1_{\mathcal P},X;q,a)-\Delta(f\cdot 1_{\mathcal P},y;q,a) |  \ll_A \frac X {(\log x)^C}
\end{equation}
for all $X$ in the range $y\leq X\leq x$.
 
Let $y=x/ (\log x)^{\gamma}$ and select any $Q$ in the range $x^{1/3}<Q\leq x^{2/5}$.  Let $\mathcal P$ be the set of primes $p$ in the range $y/2<p\leq y$ for which 
 there exists a prime $q \in (Q, 2Q]$ that divides $p-1$.  We will work with the completely multiplicative function $f$, defined as follows: 
 \[
 f(p) = \begin{cases}
 0 & \text{ if } p\leq 2(\log x)^{\gamma} \text{ or } y<p\leq x;\\
 -1& \text{ if }  p \in \mathcal P;\\
  1& \text{ otherwise}.\\
 \end{cases}
  \]
Since $f$ is supported only on $y$-smooth integers, \eqref{eq:BVI4}  trivially holds for all $X$ in the range $y\leq X\leq x$. From now on we follow the arguments of section 8.2 of \cite{GSh}, and assume (1.5) of \cite{GSh}) (which is a strong form of the Prime Number Theorem in arithmetic progressions). First we may deduce that $f$ satisfies the Siegel-Walfisz criterion in the hypothesis of Theorem   \ref{thm: mr2}. Now note that
\[ f(n) = |f(n)| - 2 \cdot \1_{\mathcal P}(n) \]
for each $n \leq x$. Thus
\[ \Delta(f,x;q,1) = \Delta(|f|,x;q,1) - 2\Delta(\1_{\mathcal P},x;q,1). \]
Note that $|f|$ is the indicator function of the set of $y$-smooth integers with no prime factors  $\leq 2(\log x)^{\gamma}$. It is straightforward to establish that this set has level of distribution $x^{1/2-\ee}$. Thus
\[ \sum_{Q < q \leq 2Q} |\Delta(f,x;q,1)| \geq 2\sum_{Q < q \leq 2Q} |\Delta(\1_{\mathcal P},x;q,1)| - O\left(\frac{x}{(\log x)^{\gamma+3}}\right). \]
On the other hand, we have
\[ \Delta(\1_{\mathcal P},x;q,1) = \sum_{\substack{y/2<p\leq y \\ p\equiv 1\pmod{q}}} 1 - \frac{\#\mathcal P}{\varphi(q)}, \]
for prime $q \in (Q, 2Q]$, where, by the definition of $\mathcal P$, we are able to extend the range for the first summation from $p \in \mathcal P$ to all primes in $(y/2,y]$. By the Brun-Titchmarsh inequality, we have $\#\mathcal P \ll y/(\log x)^2$. Thus
\[
|\Delta(\1_{\mathcal P},x;q,1)| \gg  \frac{y}{\varphi(q)\log x} - O\left(\frac{y}{\varphi(q)(\log x)^2}\right) \gg \frac{y}{\varphi(q)\log x}.
\] 
Summing this over all primes $q \in (Q, 2Q]$, we obtain
\[ 
\sum_{Q < q \leq 2Q}   |\Delta(f,x; q,1)| \gg   \frac y{(\log x)^2 } - O\left(\frac{x}{(\log x)^{\gamma+3}}\right) \gg   \frac x{(\log x)^{\gamma+2} } .
\]
 Therefore if this is $\ll x/(\log x)^{A-1}$, we must have $\gamma \geq A-3$, as claimed in the remarks following Theorem  \ref{thm: mr2}.

 \section{Further thoughts} 
 
 Arguably the most intriguing issue is to try to improve the exponent, $\gamma$, of the logarithm in the definition of $y$ in Theorem \ref{thm: mr2}. We have shown that $A-3\leq \gamma(A)\leq 2A+6+\epsilon$; one might guess that the optimal exponent has
 $\gamma(A)=\kappa A+O(1)$ for all $A\geq 0$. 
 
 The bound $\kappa\leq 2$ on the coefficient $\kappa$ is a consequence of the $(xy)^{1/2}(\log x)^{O(1)}$-term
  in the upper bound in Proposition \ref{prop:factorize2}. If one can replace this term by $y (\log x)^{O(1)}$ then $\kappa = 1$ follows. 
  Extending the idea in the proofs of Proposition \ref{prop:factorize2} and Proposition \ref{prop:factorize}, one 
  can restrict attention to a much smaller class of $f$: Those completely multiplicative $f\in \mathcal C$ that are supported only on the primes $\leq 2(\log x)^\gamma$, and the primes in $(y/2,y]$.


\end{document}